\documentclass[12pt]{amsart}

\usepackage{latexsym, amssymb, amsmath,ulem,soul,esint}

\usepackage{amsfonts, graphicx}
\usepackage{graphicx,color}
\newcommand{\be}{\begin{equation}}
\newcommand{\ee}{\end{equation}}
\newcommand{\beq}{\begin{eqnarray}}
\newcommand{\eeq}{\end{eqnarray}}

\newtheorem{thm}{Theorem}[section]

\newtheorem{lma}{Lemma}[section]
\newtheorem{prop}{Proposition}[section]
\newtheorem{cor}{Corollary}[section]

\theoremstyle{remark}
\newtheorem{rem}{Remark}[section]
\numberwithin{equation}{section}
\def\tr{\operatorname{tr}}

\def\be{\begin{equation}}
\def\ee{\end{equation}}
\def\bee{\begin{equation*}}
\def\eee{\end{equation*}}
\def\ol{\overline}
\def\lf{\left}
\def\ri{\right}
\def\div{\mathrm{div}}

\def\wh{\widehat}

\def\mS{\mathbb{S}}

\def\cR{\mathcal{S}}

\def\Ric{\text{\rm Ric}}

\def\cS{\mathcal{S}}

\def\wt{\widetilde}
\def\la{\langle}
\def\ra{\rangle}
\def\p{\partial}

\def\ol{\overline}

\def\tr{\operatorname{tr}}

\def\e{\varepsilon}
\def\a{{\alpha}}
\def\b{{\beta}}

\def\R{\mathbb{R}}

\begin{document}
\title[]
{Singular metrics with negative scalar curvature}

\author{ Man-Chuen Cheng}
\address[Man-Chuen Cheng]{Department of Mathematics, The Chinese University of Hong Kong, Shatin, Hong Kong, China.}
\email{mccheng@math.cuhk.edu.hk}

 \author{Man-Chun Lee}
\address[Man-Chun Lee]{Department of Mathematics, Northwestern University, 2033 Sheridan Road, Evanston, IL 60208}
\curraddr{Mathematics Institute, Zeeman Building,
University of Warwick, Coventry CV4 7AL}
\email{mclee@math.northwestern.edu, Man.C.Lee@warwick.ac.uk}
\thanks{$^1$Research partially supported by EPSRC grant number P/T019824/1.}

\author{Luen-Fai Tam}
\address[Luen-Fai Tam]{The Institute of Mathematical Sciences and Department of Mathematics, The Chinese University of Hong Kong, Shatin, Hong Kong, China.}
 \email{lftam@math.cuhk.edu.hk}
\thanks{$^2$Research partially supported by Hong Kong RGC General Research Fund \#CUHK 14300420}

\renewcommand{\subjclassname}{
  \textup{2010} Mathematics Subject Classification}

\date{\today}

\begin{abstract}
{ Motivated by the work of Li and Mantoulidis \cite{LiMantoulidis2019},     we study singular metrics which are uniformly Euclidean $(L^\infty)$ on a compact manifold $M^n$ ($n\ge 3$) with negative Yamabe invariant $\sigma(M)$.  It is well-known that if $g$ is a smooth metric on $M$ with unit volume and with scalar curvature $\cR(g)\ge \sigma(M)$, then $g$ is Einstein.  We show,  in all dimensions,  the same is true for metrics with  edge singularities with cone angles $\leq 2\pi$ along codimension-2 submanifolds. We also show in three dimension, if  the Yamabe invariant of  connected sum of two copies of $M$ attains its minimum, then  the same is true for $L^\infty$ metrics with isolated point singularities.}

\end{abstract}


\maketitle

\markboth{Man-Chuen Cheng, Man-Chun Lee and Luen-Fai Tam}{Singular metrics with negative scalar curvature}
\section{Introduction}\label{s-intro}

Let $M^n$ be a smooth compact oriented manifold. In this work, we always assume that $n\ge 3$. Let $\mathcal{C}$ be a conformal class of smooth Riemannian metrics on $M$, the {\it Yamabe constant of $\mathcal{C}$} is defined as:
$$
\mathcal{Y}(M,\mathcal{C})=\inf_{g\in\mathcal{C}}\displaystyle{\frac{\int_M\mathcal{S}_gdv_g}{(V(M,g))^{1-\frac2n}}},
$$
where $\mathcal{S}_g$ is the scalar curvature, $dV_g$ is the volume element and $V(M,g)$ is the volume of $M$ with respect to $g$. The {\it Yamabe invariant} introduced by Schoen \cite{Schoen1989}, see also Kobayashi \cite{Kobayashi1987},   is defined as
$$
\sigma(M)=\sup_{\mathcal{C}}\mathcal{Y}(M,\mathcal{C}).
$$
The supremum is taken among all conformal classes of smooth metrics on $M$.

 In case $\sigma(M)\le 0$,  it is known that if $g$ is a smooth metric on $M$ with unit volume and $\cR(g)\ge \sigma(M)$, then $g$ is Einstein with scalar curvature $\sigma(M)$. In case that $g$ is non-smooth,   in \cite{LiMantoulidis2019}, among other things Li and Mantoulidis proved the followings:

{\it
 \begin{enumerate}
   \item [(i)] \underline{Singularity of co-dimension two}:
 Let $M^n$ be a compact manifold with $\sigma(M)\le0$ and let $N$ be an embedded two sided submanifold of co-dimension two. Suppose $g$ is a metric on $M$ with $g\in C^\infty(M\setminus N)\cap L^\infty(M)$ so that $g$ has cone singularity near $N$ with certain order (see the definitions below). Suppose the $\cR(g)\ge0$ in $M\setminus N$, then $g$ can be extended smoothly to the whole manifold and $g$ is Ricci flat.

   \item [(ii)] \underline{Singularity of co-dimension three}: As in (i), suppose $\sigma(M)\le 0$,  $n=3$ and $N$ consists of finitely many points. If $g\in C^\infty(M\setminus N)\cap L^\infty(M)$ such that $\cS(g)\ge 0$ in $M\setminus N$, then $g$ can be extended smoothly to the whole manifold and $g$ is flat.
 \end{enumerate}
}

\bigskip

On the other hand, Shi and the second author \cite{ShiTam2018} studied similar questions under different assumptions. For example, they proved that if $\sigma(M)\le 0$ and if $g$ is a metric on $M$ with unit volume so that $g$ is smooth except on a set of codimension 2, $g$ is in $W^{1,p}$ for some $p>n$, with $\cR(g)\ge \sigma(M)$ in $M\setminus S$, then $g$ is Einstein on $M\setminus N$. Since $\sigma(M)$ may be negative, $\cR(g)$ could be negative.
Hence it is natural to ask if  results similar to (i) and (ii) are still true in the setting of $\sigma(M)<0$.

 First, let us   recall the notion of cone singularity defined as in \cite{LiMantoulidis2019} which was originated by  Atiyah-Lebrun \cite{AtiyahLebrun}. Let $M^n$ be a compact manifold without boundary with metric $g$ so that it is smooth on $M\setminus N$ where $N$ is an embedded two-sided submanifold of co-dimension 2. In a neighbourhood $U$ of $N$, $g$ is of the form:
\bee
g=dr^2+\b^2 r^2(d\theta+\zeta)^2+\omega+r^{1+\eta}h
\eee
on $U\setminus N$. Here we assume
\begin{itemize}
  \item $0<r<r_0$.
  \item $\b$ is a smooth function on $N$ so that
  $$
  0<\inf_N\b \le \sup_N\b\le 1.
  $$
  \item $\omega$ is smooth metric on $N$.
  \item $\zeta$ is smooth 1-form on $N$ with $||\zeta||_{C^2(N,\omega)}<\infty$.
  \item $h$ is a smooth symmetric two tensor on $U$ with $||h||_{C^2(U,\mathfrak{h})}<\infty$ where { $\mathfrak{h}$  is a smooth background metric on $M$.}
      \item $\eta>1$ is a constant. (This implies that   $r^{1+\eta}h$ is in $C^{2, \eta-1}$ in $U$.)

\end{itemize}
In this case, we say that $g$ has {\it cone singularity} at $N$ of order $\eta$.
Also, a measurable metric $g$ is said to be in $L^\infty(M)$ if
$$
\Lambda^{-1}\mathfrak{h} \le g\le \Lambda \mathfrak{h}
$$
for some smooth background metric $\mathfrak{h}$ for some positive constant $\Lambda\geq 1$.
\vskip .1cm

We have the following:

\begin{thm}\label{t-cone}
Let $M^n$ be a compact manifold without boundary so that  the Yamabe invariant $\sigma(M)$ of $M$ is negative. Let $g\in L^\infty(M)$ be a metric with cone singularity at an embedded two-sided submanifold $N$ with order $\eta>2-\frac 4n$ and $\eta> 1$ in case $n=3$ so that $V_g(M)=1$, and the scalar curvature satisfies $\cR(g)\ge \sigma$ on $M\setminus N$. Suppose $g\in L^\infty(M)$, then $g$ can be extended to a $C^{2,\gamma}$ metric on $M$  for some $\gamma>0$ so that $g$ is Einstein { with scalar curvature $\sigma$}.
\end{thm}
One crucial step in the proof is to smooth $g$ as   in \cite{LiMantoulidis2019}, which was motivated by the work of Miao \cite{Miao2002}. Unlike the case that the scalar curvature is nonnegative, it is important to normalize the metric so that the volume is unity. Hence  in order to prove the theorem, one has to keep track the volume when we try to smooth the metric. One then uses the method of Trudinger \cite{Trudinger1968} to obtain another smooth metric with scalar curvature larger than $\sigma(M)$ if the metric is not smooth or Einstein.

Corresponding to the result (ii) of \cite{LiMantoulidis2019}, we only have  partial results. Let $M^3$ be a compact manifold without boundary with $\sigma=:\sigma(M)<0$. Suppose $A$ is a finite set of points in $M$. { By Kobayashi \cite[Theorem 2]{Kobayashi1987}, it is known that the connected sum  $M\#M$ of two copies of $M$ along $A$ satisfies}
$$
\sigma(M\# M)\ge 2^\frac23\sigma.
$$
We have the following:

\begin{thm}\label{t-point-removable}
Let $M^3$ be a compact manifold with $\sigma=:\sigma(M)<0$. Let $A$ be a finite set of points in $M$ and   let $M\#M$ be the connected sum of two copies of $M$ along $A$. Suppose $g$ is a metric on $M$ with unit volume so that $g\in C^\infty(M\setminus A)\cap L^\infty(M)$ and   the scalar curvature of $g$ in $M\setminus A$ is at least $\sigma$.
  If $\sigma(M\# M)=2^\frac23\sigma$, then $g$ can be extended to  $A$ so that $g$ is a $C^{2,\a}$ metric with constant { scalar curvature $\sigma$.}
\end{thm}
We basically follows the strategy in \cite{LiMantoulidis2019}. However, their method cannot be carried over directly. The difficulties are to find stable minimal spheres and try to control the volume of some smoothing metric. See the proof of the theorem in \S \ref{s-point} for more details. The result (ii)   of Li-Mantoulidis  can be rephrased as follows: Let $M^3$  be a compact three manifold. Suppose there is a metric $g$ on $M$ which is in $C^\infty(M\setminus A)\cap L^\infty(M)$ for some finite set $A$ so that $\cR(g)\ge0$ { on $M\setminus A$ and $\cR(g)>0$ somewhere}, then $\sigma(M)>0$. In this form, we have another partial result:

\begin{prop}\label{p-point} Let $(M^3,g)$ be a compact manifold without boundary and $g\in C^\infty(M\setminus A)\cap L^\infty(M)$ for some finite set $A$ with unit volume. Suppose $\cR(g)\ge \sigma$ with $\sigma<0$ is a constant and suppose $\cR(g)>0$ near any point in $A$. $\sigma(M)>\sigma$.
\end{prop}

The  paper is organized as follows.  We will prove Theorem \ref{t-cone} in \S\ref{s-cone}, and will prove Theorem \ref{t-point-removable} and Proposition \ref{p-point} in \S\ref{s-point}. In \S\ref{s-eigenvalue}, we will study the connectivity of metrics with fixed area on $\mS^2$ so that the first eigenvalue of $-\Delta+K$ is bounded below by a nonpositive constant. This question is related to the proof of (ii) in \cite{LiMantoulidis2019} and similarly to the proof of Proposition \ref{p-point}. This is also related to the estimate of the generalized Bartnik mass by Cabrera Pacheco,   Cederbaum and McCormick \cite{CCM2018}. In \S\ref{s-variational} we will sketch a proof of a modified version of Trudinger's  theorem \cite{Trudinger1968}.  In the appendix, we sketch a  proof  the existence of Green's function, which   will be used to obtain the main results of the paper.

The authors would like to thank Kwok-Kun Kwong and Andrejs Treibergs for some useful discussion.

\section{A variational problem}\label{s-variational}

In this section, we will give a proof of the following result which is modified from the celebrated work of Trudinger  \cite[Theorem 2]{Trudinger1968}. { The result will be used later.}

\begin{thm}\label{t-Trudinger}
Let $M^n$ be a compact manifold without boundary with a smooth background metric $\mathfrak{h}$. Suppose $g$ is a $C^{2,\theta}$ metric for some $\theta>0$ such that
$$
\Lambda^{-1} \mathfrak{h}\le g\le \Lambda\mathfrak{h}
$$
for some positive constant $\Lambda>0$ and such that $V_g(M)=1$. For any $\a>0$ and $Q>\frac n2$, there is $\e>0$ depending only on $n, \a, \Lambda, Q$ and $\mathfrak{h}$ such that if   $f$ is a measurable function on $M$ satisfying
$$
||f||_Q=\lf(\int_M |f|^Q dV_g\ri)^\frac1Q\le \a,
$$
and
$$
\int_M fdV_g\le \e,
$$
then there is a $W^{1,2}$ function $u\ge0$ which is a weak solution of
\bee
-a\Delta_gu+fu=\lambda u^{P-1}
\eee
where $
P= \frac{2n}{n-2}, \  a=\frac{4(n-1)}{n-2}
$
and $\lambda=\int_M (a|\nabla_g u|^2+fu^2) dV_g$. Moreover, $ u\in L^P$ with $\int_M u^P dV_g=1$, and $C^{-1}\le u\le C$ for some positive constant $C,\gamma$ depending only on $n, \a, \Lambda, Q,\mathfrak{h}$.
\end{thm}

The proof follows verbatim from that in \cite{Trudinger1968}. For the sake of completeness, we will give a sketch here { so that we can keep track of the constants involved.}

 In the following, all norm and connection are with respect to $g$. Let $2<q<P$ and consider the functional
$$F_q(u)=\int_M a |\nabla u|^2+fu^2 \;dV_g$$
subject to $\int_M|u|^qdV_g=1$. For notational convenience, denote
$$W^{1,2}_q=\left\{ u \in W^{1,2}: \int_M |u|^q dV_g =1\right\}$$
and $\lambda_q=\inf_{u\in W^{1,2}_q}F_q(u)$.

\begin{lma}\label{t-Trudinger-lma-0}
If $f\in L^Q$ for some $Q>n/2$, then {$\lambda_q$ is finite} for all $\frac{2Q}{Q-1}\leq q<P$.
\end{lma}
\begin{proof}
Let $u\in W^{1,2}_q$, then
\begin{equation}
\begin{split}
F_q(u)&\geq \int_M f u^2 dV_g \\
&\geq -\left(\int_M |f|^Q dV_g \right)^{1/Q} \left(\int_M |u|^\frac{2Q}{Q-1}dV_g \right)^{1-1/Q}.
\end{split}
\end{equation}
Therefore, if $\frac{2Q}{Q-1}\leq q<P$, then $F_q(u)\geq -C$ for all $u\in W^{1,2}_q$ for some uniform $C>0$. This implies a lower bound of $\lambda_q$. {The upper bound follows by taking $1$ as a test function. This completes the proof.}
\end{proof}

\begin{lma}\label{t-Trudinger-lma-1}
If $f\in L^Q$ for some $Q>n/2$, then for all $\frac{2Q}{Q-1}\leq q<P$, there is $U_q\in W^{1,2}_q$ minimizing $F_q$ in $W^{1,2}_q$ so that
\begin{enumerate}
\item[(i)] $C_q^{-1}\leq U_q\leq C_q$ for some $C_q>1$;
\item[(ii)] $U_q$ solves $-a\Delta U_q +fU_q=\lambda_q U_q^{q-1}$ weakly;
\item[(iii)] $ -||f||_Q\leq \lambda_q\leq \int_M f dV_g\leq ||f||_Q$.
\item[(iv)]$||U_q||_{W^{1,2}}\leq A(n,\a,\mathfrak{h},\Lambda)$.
\item[(v)] $||U_q||_{L^P}\leq A(n,\a,\mathfrak{h},\Lambda)$.
\end{enumerate}
{ Here $A>0$ is a constant depending only on $ n,\a,\mathfrak{h},\Lambda.$}
\end{lma}
\begin{proof}
Since $\lambda_q$ is finite, we can find a minimizing sequence $u_k$ such that $F_q(u_k)\to \lambda_q$. We may assume $u_k\geq 0$ by considering $|u_k|$. We may assume $u_k\to U_q$ weakly in $W^{1,2}$. Then Sobolev embedding implies that the convergence is in $L^Q$ and hence the $u_k\to U_q$ strongly by the same argument in \cite[Page 268]{Trudinger1968}. This implies $U_q\geq 0$ and $U_q\in W^{1,2}_q$.

Since $U_q$ is the minimizer for $\lambda_q$. For $v\in W^{1,2}$, we can consider the first order variation of $F_q(v_t)$ {at $t=0$} where $v_t=\frac{U_q+tv}{||U_q+tv||_q}$ to see that $U_q$ is a weak solution of
$$-a\Delta U_q+fU_q=\lambda_q U^{q-1}.$$
$U_q\ge0$. Since $\frac{q}{q-2}>\frac{n}{2}$ and $U_q\in L^q$, $U^{q-2}\in L^r$ for some $r>n/2$ { so that $||U_q||_q, ||U^{q-2}||_r$ are bounded by a constant depending only on $q, r, n$. Since $||U_q||_q=1$, by the Harnack inequality as in \cite{Trudinger1968},  these prove  (i) and (ii)}.

To prove (iii),   taking $u=1$ as a test function, we immediately see that
$$\lambda_q\leq \int_M f dV_g\leq ||f||_Q.$$

On the other hand,
\begin{equation}
\begin{split}
\lambda_q &=\int_M a|\nabla U_q|^2 +fU_q^2 dV_g \\
&\geq -||U_q||_q^2 ||f||_{\frac{q}{q-2}}\\
&\geq -||f||_Q.
\end{split}
\end{equation}

(iv) follows easily from the upper bound of $\lambda_q$ and the fact that $\int_M U^q dV_g=1$. (v) follows from (iv) and the Sobolev embedding theorem.
\end{proof}
Now we want to let $q\to P$ to obtain a weak solution to the desired equation.
\begin{lma}
There is $\e=\e(n,\mathfrak{h},\Lambda,\a,Q)>0$ such that { if $\int_M fdV_g<\e$},  the functional $F_P$ is minimized by a function $U\in W^{1,2}_N$. Moreover,  $U$ is a weak solution to
$$-a\Delta U+fU=\lambda U^{P-1}$$
and satisfies $C^{-1}\leq U\leq C$ and $U\in C^\theta$ for some $C=C(n,\mathfrak{h},\Lambda,\a,Q)>0$ and $ \theta(n,\mathfrak{h},\Lambda,\a,Q)>0$.
\end{lma}
\begin{proof}
Let $\frac{2Q}{Q-1}\leq q<P$ and $U_q$ be the solution obtained from Lemma~\ref{t-Trudinger-lma-1}. Since $U_q>0$ and is bounded, by the estimates from  Lemma~\ref{t-Trudinger-lma-1}, $U_q^\b\in W^{1,2}$ for all $1<\b<P-1$.

For $1<\b<P-1$, by using the equation of $U_q$ and integration by part,
\begin{equation}
\begin{split}
\lambda_q \int_M U_q^{q-1+\b}&=\int_M a\b U_q^{\b-1}|\nabla U_q|^2 +fU_q^{\b+1}\\
&=\int_M \frac{4a\b}{(\b+1)^2}|\nabla U_q^\frac{\b+1}{2}|^2 +fU_q^{\b+1}\\
\end{split}
\end{equation}

Apply Sobolev inequality, if $\int_M fdV_g<\e$, then $\lambda_q<\e$ by Lemma \ref{t-Trudinger-lma-1} and hence,
\begin{equation}
\begin{split}
||U_q^\frac{\b+1}2||_{L^P}^2
&\leq C_1\left(\int_M \frac{4a\b}{(\b+1)^2}|\nabla U_q^\frac{\b+1}{2}|^2+U_q^{\b+1} \right)\\
&= C_1\left( \lambda_q\int_M U_q^{q-1+\b}-fU_q^{\b+1}+U_q^{\b+1} \right)\\
&\leq C_1 \left( \e\int_M U_q^{q-1+\b}-fU_q^{\b+1}+U_q^{\b+1} \right)
\end{split}
\end{equation}
for some $C_1(n,\mathfrak{h},\Lambda,\b)>0$. Since $Q>n/2$, by (v) of Lemma~\ref{t-Trudinger-lma-1}, we can choose $\b$ sufficiently close to $1$ depending only on $n,Q$ so that
$$\left|\int_M  f U_q^{\b+1}dV_g\right|+\left|\int_M U_q^{\b+1}dV_g\right|\leq C_2,$$
for some $C_2(n,\mathfrak{h},\Lambda,\b,\a,Q)>0$.
At the same time,
\begin{equation*}
\int_M U_q^{q-1+\b}dV_g =\int_M U_q^{q-2} U_q^{\b+1}dV_g\leq  C_3 ||U_q^\frac{\b+1}{2}||_P^2
\end{equation*}
for some $C_3(n,\mathfrak{h},\Lambda,\b,\a,Q)>0$. Therefore, if $\e<\frac1{2C_1C_3}$, then we have
$$||U_q^\frac{\b+1}{2}||_P\leq C_4(n,\mathfrak{h},\Lambda,\b,\a,Q).$$
In particular, $U_q^{q-2}$ is bounded in $L^r$ for some $r>n/2$. As $U_q\in W^{1,2}_q$ is the weak solution to the operator $-a\Delta +(f-\lambda_q U_q^{q-2})$, we therefore conclude that
$$C_5^{-1}\leq U_q\leq C_5$$
for some uniform $C_5(n,\mathfrak{h},\Lambda,\b,\a,Q)>1$ and hence $U_q$ is bounded in $C^{\theta_1}$ for some $\theta_1(n,\mathfrak{h},\Lambda,\b,\a,Q)>0$. Since $U_q$ is bounded in $W^{1,2}$, we may let $q\to P$ to obtain $U=\lim_{q\to P}U_q$. By the H\"older estimate of $U_q$, the convergence is in $C^\theta$ for $\theta<\theta_1$. By Lemma \ref{t-Trudinger-lma-1}(ii),  we conclude that $U$ satisfies weakly
$$-a \Delta U+fU=\lambda U^{P-1}$$
where $\lambda=\lim_{q\to P} \lambda_q$. This completes the proof.
\end{proof}

{
\begin{rem}
It is clear from the proof that Theorem~\ref{t-Trudinger} is still true if $g\in L^\infty(M)\cap C^{2,\theta}_{loc}(M\setminus S)$ where $S\subset M$ is compact with $\mathrm{Vol}(S)=0$. This is analogous to \cite[Lemma 4.1]{LiMantoulidis2019} where they studied the first eigenfunction of $-a\Delta+f$ instead. Using this, it might be possible to cover the case when $g$ is singular on a compact non-degenerate 1-skeleton as discussed in \cite{LiMantoulidis2019}.
\end{rem}

}

\section{Cone singularity}\label{s-cone}

In this section, we want to prove Theorem \ref{t-cone}. Let $M, N, g$ be as in the theorem so that
$$
\Lambda^{-1}\mathfrak{h} \le g\le \Lambda \mathfrak{h}
$$
for some smooth background metric $\mathfrak{h}$ and some positive constant $\Lambda$.
The first step is to  smooth $g$ as in  \cite{LiMantoulidis2019}. More precisely, take a smooth function $\phi$ with
 \bee
 \left\{
   \begin{array}{ll}
     \phi=0, & \hbox{on $[0, \frac13]$;} \\
     \phi=1, & \hbox{on $[\frac 23,1]$;} \\
     0\le \phi'\le 6;\\
\phi'=1, & \hbox{on $[\frac49, \frac59]$.}
   \end{array}
 \right.
 \eee
 For any $1\ge \e>0$ let
$$
f_\e(r,y)=1+\phi\lf(\frac r\e\ri)\lf( \frac1{\b(y)}-1\ri)
$$
for $y\in N$. Then
\be\label{e-f}
\left\{
  \begin{array}{ll}
    1\le f_\e\le C, & \hbox{for some $C>0$ independent of $\e$;} \\
    \p_rf_\e=
     0, & \hbox{on $N_{\frac13\e}$ and $N_\e\setminus N_{\frac 23\e}$;} \\
     \p_rf_\e=
\e^{-1}(\b^{-1}-1), & \hbox{on $N_{\frac 59\e}\setminus N_{\frac 49\e}$;}\\
0\le \p_rf_\e\le  6\e^{-1}(\b^{-1}-1),
  \end{array}
\right.
\ee
where $N_\e=\{x\in U|\ r(x)<\e\}$, etc.
$f_\e\ge1$  because $\b\le 1$. Moreover, $  f_\e\le C$ for some constant $C$ independent of $\e$ because $\b$ is bounded away from zero. Let
\be
g_\e=\b^2\lf(f_\e^2dr^2+r^2(d\theta+\sigma)^2+\b^{-2}\omega+f_\e^2r^{1+\eta}h\ri),
\ee
for $r<\e $ and let $g_\e=g$ outside $N_\e$.
Then $g_\e$ is $C^{2,\gamma}$ for some $\gamma>0$ independent of $\e$, because $\eta>1$.

\begin{lma}\label{l-smoothing-1}
\bee\label{e-smooth-1}
\left\{
  \begin{array}{ll}
    g_\e=g, & \hbox{outside $N_{\e}$;} \\
 ( \Lambda')^{-1}\mathfrak{h}\le g_\e\le   \Lambda' \mathfrak{h}, & \hbox{for some $\Lambda'>0$ independent of $\e$;} \\
    |V_{g_\e}(M)-1|\le C\e^2 & \hbox{for some $C>0$ independent of $\e$};\\
|\cR(g_\e)-2\b^{-2}r^{-1}f_\e^{-3}\p_r f_\e|\le C (1+r^{-2+\eta}) &\hbox{in $N_\e$ for some $C>0$ independent of $\e$} \end{array}
\right.
\eee
where $\cR(g_\e)$ is the scalar curvature of $g_\e$, provided $r<r_0$ for some $r_0>0$ independent of $\e$.
\end{lma}
\begin{proof} The conclusions of the lemma  are obvious except for the estimates of the scalar curvature. Let
$$
\wt g_\e=f_\e^2dr^2+r^2(d\theta+\zeta)^2+\b^{-2}\omega.
$$
By the proof of \cite[Proposition 3.2]{LiMantoulidis2019}, we have
\bee
 \lf|\cR(\wt g_\e)-2r^{-1}f_\e^{-3}\p_r f_\e\ri|\le C_1
\eee
for $r\le r_1$ with $r_1>0$ independent of $\e$.
Here and below, $C_i$ will denote a constant which is independent of $\e$.
On the other hand, {
\bee
\begin{split}
\cR(\b^2\wt g_\e)=&\b^{-\frac{n+2}2}\lf(-a\Delta_{\wt g_\e}+\cR(\wt g_\e)\ri)\b^\frac {n-2}2\\
=&\b^{-2}\cR(\wt g_\e)-a\b^{-\frac{n+2}2} \Delta_{\wt g_\e} \b^\frac {n-2}2
\end{split}
\eee
where $a=\frac{4(n-1)}{n-2}$.}
Since $\b$ depends only on $y$, by the condition on $\zeta$, we conclude that
$$
\lf|\cR(\b^2\wt g_\e)-2\b^{-2}r^{-1}f_\e^{-3}\p_r f_\e\ri|\le C_2
$$
for $r\le r_1$.
Now
\bee
g_\e=\b^2\wt g_\e+\b^2f_\e^2 r^{1+\eta}h.
\eee
As in \cite{LiMantoulidis2019},
$$
 \lf|\cR(\b^2\wt g_\e)-\cR(g_\e)\ri|\le C_3r^{-2+\eta}
 $$
 Hence the result follows.
\end{proof}

 \bigskip
By \eqref{e-f}, the  more detailed information on $\cR(g_\e)$ is as follows. Inside $N_\e$
\be\label{e-scalar-ge-1}
-C(1+r^{-2+\eta})\le \cR(g_\e)\le 12\b^{-2}r^{-1} \e^{-1} (\b^{-1}-1)+ C(1+r^{-2+\eta})
\ee
for some constant $C>0$ independent of $\e$.
\be\label{e-scalar-ge-2}
\left\{
  \begin{array}{ll}
    \cR(g_\e)=\cR(g), & \hbox{on $M\setminus N_\e$;} \\
-C_1(1+r^{-2+\eta}) \le    \cR(g_\e)\le C_1(1+r^{-2+\eta}), &   \hbox{on $N_{\frac13\e}$; and $N_\e\setminus N_{\frac 23\e}$;} \\
   C_2^{-1}\e^{-2}(\b^{-1}-1)- C_1(1+r^{-2+\eta})\le \cR(g_\e),\\
   \cR(g_\e) \le C_2\e^{-2}(\b^{-1}-1)+ C_1(1+r^{-2+\eta}) & \hbox{on $N_{\frac 59\e}\setminus N_{\frac 49\e}$.}
  \end{array}
\right.
\ee
if $\e$ is small enough, where $C_1, C_2>0$ are constants independent of $\e$. \bigskip

{ We will need the following Lemma which was proved originally by LeBrun in \cite[Proposition 4.1]{Lebrun1999}. }
\begin{lma}\label{l-Lebrun}
Let $M^n$ be a compact manifold without boundary so that $\sigma(M)\le 0$. Then $\sigma(M)$ can be characterized as follows:
\bee
|\sigma(M)|=\inf_{h} \lf(\int_M (\cR(h)_-)^\frac n2dV_h \ri)^\frac 2n
\eee
where the infimum is taken over all smooth metrics $h$ on $M$. Here $\cR(h)_-=-\min\{\cR(h),0\}$ is the negative part of the scalar curvature $\cR(h)$ of $h$.
\end{lma}
\begin{proof}

Let $M^n$ be a closed manifold.   Let $\mathcal{C}$ be a conformal class of metrics. Then there is a metric $g\in \mathcal{C}$ with unit volume so that the scalar curvature of $g$ is the constant $\mathcal{Y}(M,\mathcal{C})$ by \cite{Trudinger1968}, see also \cite{Aubin1976,Schoen1984}.   Then
\bee
\begin{split}
 |\mathcal{Y}(M,\mathcal{C})|=&|\int_M \cR(g) dV_g|\\
 =&\int_M |\cR(g)_-|dV_g\\
 =&\lf(\int_M|\cR(g)_-|^\frac n2 dV_g\ri)^\frac 2n\\
 \ge& \inf_{h}\lf(\int_M|\cR(h)_-|^\frac n2 dV_h\ri)^\frac 2n
 \end{split}
\eee
because $\cR(g)=\mathcal{Y}(M,\mathcal{C})$ is a constant. From this we have:
$$
|\sigma(M)|\ge \inf_{h}\lf(\int_M|\cR(h)_-|^\frac n2 dV_h\ri)^\frac 2n.
$$

On the other hand, let $g$ be any metric and let $\mathcal{C}$ be the conformal class of $g$. Then there is a positive function $u$ with
$$
\int_M u^N dV_g=1
$$
where $N=2n/(n-2)$ and
\bee
\begin{split}
\mathcal{Y}(M,\mathcal{C})=&\int_M\lf( a|\nabla u|^2+\cR(g) u^2 \ri)dV_g\\
\ge &\int_M -\cR(g)_-u^2 dV_g
\end{split}
\eee
where $a=4(n-1)/(n-2)$. Hence
\bee
\begin{split}
|\sigma(M)|&\le
|\mathcal{Y}(M,\mathcal{C})| \\
 & \le\int_M  \cR(g)_-u^2 dV_g\\
&\le \lf(\int_M\cR(g)_-^\frac n2dV_g\ri)^\frac 2n \lf(\int_M u^N dV_g\ri)^\frac{n-2}n\\
&\le \lf(\int_M\cR(g)_-^\frac n2dV_g\ri)^\frac 2n.
\end{split}
\eee

So
$$
|\sigma(M)|\le \inf_{h}\lf(\int_M|\cR(h)_-|^\frac n2 dV_h\ri)^\frac 2n.
$$
Hence the lemma is true.
\end{proof}

It is easy to see that { by approximation the lemma is still true if the infimum is taken over all $C^2$ metrics which are H\"older in the second order derivatives.}

\vskip.1cm

  Now we are read to prove Theorem \ref{t-cone}.
\begin{proof}[Proof of Theorem \ref{t-cone}] For any $1\ge \e>0$ small enough, let $\wt g_\e=v_\e^{-\frac n2}g_\e$ where $v_\e=V_{g_\e}(M)$, so that $V_{\wt g_\e}(M)=1$ and $\cR(\wt g_\e)=v_\e^{\frac n2}\cS(g_\e)$. For notational convenience, we will still denote $\wt g_\e$ by $g_\e$.  By Lemma \ref{l-smoothing-1},
\be\label{e-smooth-2}
\left\{
  \begin{array}{ll}
    g_\e=v_\e^{ -\frac{2}{n}}  g, & \hbox{outside $N_{\e}$;} \\
 ( \Lambda)^{-1}\mathfrak{h}\le g_\e\le   \Lambda \mathfrak{h}, & \hbox{for some $\Lambda>0$ independent of $\e$;} \\
 |v_\e-1|\le C\e^2 & \hbox{for some $C>0$ independent of $\e$};\\
|\cR(g_\e)-2v_\e^{\frac n2}\b^{-2}r^{-1}f_\e^{-3}\p_r f_\e|\le C (1+r^{-2+\eta}) &\hbox{in $N_\e$ for some $C>0$ independent of $\e$.} \end{array}
\right.
\ee

\bigskip
 We want to solve
$$
-a\Delta_g u+\kappa u=\lambda u^P
$$
for some $\kappa$. To apply Theorem \ref{t-Trudinger}, fix $\lambda_0>0$ so that
$$
\eta>\max\left\{1,2-\frac4{(1+\lambda_0)n}\right\}.
$$
Let $q$ be such that $\frac n2<q<\frac{(1+\lambda_0)n}2.$ Then
$$
\eta>2-\frac 2q.
$$
In order to apply Theorem \ref{t-Trudinger}, we have to modify $\cR(g_\e)$ as in \cite{LiMantoulidis2019}. Namely, let
$\wt\rho$ be a smooth function on $\R$ so that $\wt\rho(t)=t$ for $t\le 1$, $\wt\rho(t)=2$ for $t\ge 3$, so that $\wt\rho'\ge0$, and $\wt\rho(t)\le t$ for all $t$. Let $\rho$ be a smooth function on $\R\times \R^+$ so that
\be\label{e-rho-1}
\rho(t,s)=s\wt \rho(\frac ts).
\ee
Then for any $s>0$,
 \be\label{e-rho-2}
 \left\{
   \begin{array}{ll}
     \rho(t,s)=t  & \hbox{for $t\le s$;} \\
     \rho(t,s)=2s, & \hbox{for $t\ge 3s$;} \\
     \rho(t,s)\le t, & \hbox{for all $t$;}\\
\p_t\rho\ge0.
   \end{array}
 \right.
\ee

 Define a smooth function $\tau$ on $M$ so that
\begin{equation}
\left\{
\begin{array}{ll}
\tau=1 &\hbox{outside } N_{2\e};\\
\tau=\e^{-\frac2q}&\hbox{on } N_{\e};\\
0<\tau\le \e^{-\frac2q} &\hbox{on } M.\\
\end{array}
\right.
\end{equation}
Fix $1\ge \delta>0$, define
$$
\kappa(x)=\rho(\cR(g_\e)(x), \delta \tau(x))
$$
Then in $M\setminus N_{2\e}$,
\be\label{e-kappa-1}
\sigma-C_1\e^2\le\kappa(x)\le 2\delta
\ee
because $\cR(g_\e)\ge \sigma-C_1\e^2$ in $M\setminus N_{2\e}$ for some constant $C_1>0$ independent of $\e$ by \eqref{e-smooth-2} {and $\sigma\leq 0$.}

Inside $N_{2\e}$, we have
\be\label{e-kappa-2}
\kappa(x)\le 2\delta \e^{-\frac 2q}.
\ee
 In $N_{2\e}\setminus N_\e$, as before we have
\be\label{e-kappa-3}
\kappa(x)\ge \sigma-C_1\e^2
\ee
 By \eqref{e-f},  inside $N_\e$
\bee
-C(1+r^{-2+\eta})\le \cR(g_\e)\le 12\b^nr^{-1} \e (\b^{-1}-1)+ C(1+r^{-2+\eta})
\eee
for some $C>0$ independent of $\e$. So
\be\label{e-kappa-4}
-C_3(1+r^{-2+\eta})\le \kappa(x)\le 2\delta\e^{-\frac 2q}.
\ee
for some $C_3>0$ independent of $\e, \delta$ by our choices of $q$ and the assumption on $\eta$.

From these we have:
\be\label{e-kappa-5}
\int_M \kappa dV_{g_\e}\le C_4\delta
\ee
for some $C_4>0$ independent of $\delta, \e$ and we have for $\frac n2<q'<q$,
\bee
\int_M |\kappa|^{q'} dV_{g_\e}\le C_3
\eee
for some constant $C_5$ independent of $\delta, \e$, which may depend on $q'$.
For such $q'$, we choose $\delta>0$ small enough so that we can apply  Theorem \ref{t-Trudinger} to  find $u>0$ such that $u$ is in $C^{2,\gamma}$ for some $\gamma>0$ and
\bee
-a\Delta_{g_\e} u+\kappa u=\lambda u^{P-1}
\eee
where
$$
\lambda=\int_M (|\nabla u|^2+\kappa u^2) dV_{g_\e}.
$$
Moreover, $C_5^{-1}\le u\le C_5$ for some $C_5>0$ independent of $\e$, and $\int_M u^P dV_{g_\e}=1$.
Then the metric
$$
 \ol g_\e= u^{\frac  4{n-2}}g_\e
$$
is a $C^{2,\gamma}$ metric for some $\gamma>0$ so that $V_{\ol g_\e}(M)=1$, and
\be\label{e-scalar-ge-3}
\begin{split}
\cR(\ol g_\e)=&u^{-\frac{n+2}{n-2}}\lf(-a\Delta u+\cR(  g_\e)u\ri)\\
=&u^{-\frac{n+2}{n-2}}\lf(-a\Delta u+ \kappa u+(\cR(  g_\e)- \kappa)u\ri)\\
\ge&\lambda
\end{split}
\ee
where we have used the fact that $\kappa\le \cR(g_\e)$.

Suppose $\b<1$ somewhere. Then there is an open set $\mathcal{O}$ in $N$ so that $\b\ge c>0$ for some $c>0$. Hence inside $\mathcal{O}_{\frac 59\e}\setminus \mathcal{O}_{\frac 49\e}$,
$$
C^{-1}\e^{-2}\le \cR(g_\e)\le C\e^{-2}
$$
for some constant $C$ independent of $\e$ provided it is small enough. We have
$$
\kappa(x)=2\delta \e^{-\frac 2q}
$$
in this set. In particular, as $\sigma\le 0$, we have
\bee
\int_M(\kappa-\sigma)_+ dV_{g_\e}\ge C_5 \e^{-\frac2q+2}
\eee
for some positive constant $C_5$ independent of $\e$, where $(\kappa-\sigma)_+$ is the positive part of $\kappa-\sigma$. On the other hand, by \eqref{e-kappa-1} and \eqref{e-kappa-3},
\bee
\int_{M\setminus N_\e} (\kappa-\sigma)_-dV_{g_\e}\le C_5\e^2,
\eee
and by \eqref{e-kappa-4},
\bee
\begin{split}
\int_{ N_\e} (\kappa-\sigma)_-dV_{g_\e}\le&C_6 \int_{N_\e}(1+r^{-2+\eta}) dV_{g_\e}\\
\le &C_7(\e^2+\e^\eta)
\end{split}
\eee
for some positive constants $C_6, C_7$ independent of $\e$. By \eqref{e-scalar-ge-3}
\bee
\begin{split}
\cR(\ol g_\e)-\sigma\ge&\lambda-\sigma\\
\ge&\int_M|\nabla u|^2+\kappa u^2 dV_{g_\e}-\sigma\int_M u^P dV_{g_\e}\\
\ge&\int_M  (\kappa-\sigma) u^2dV_{g_\e}\\
\ge&C_8\int_M (\kappa-\sigma)_+  dV_{g_\e}-C_9\int_M (\kappa-\sigma)_-dV_{g_\e}\\
\ge &C_{10} \e^{-\frac2q+2} -C_{11}(\e^2+\e^\eta)\\
>&0
\end{split}
\eee
provided $\e$ is small enough. Here we have used the facts that $\sigma\le 0$,
$$\int_M u^2dV_{g_\e}\le\int_Mu^PdV_{g_\e}=1,$$
  $C^{-1}\le u\le C$ for some positive constant $C$ independent of $\e$ and the fact that
$$
\eta> 2-\frac 2q.
$$
Hence $\ol g_\e$ is a $C^{2,\gamma}$ metric on $M$ with unit volume and with scalar curvature larger than $\sigma$. By Lemma \ref{l-Lebrun}, this is impossible. Hence we conclude that $\b=1$ everywhere and hence $g$ is $C^{2,\gamma}$ on $M$ for some $\gamma>0$. The last assertion follows from \cite[p.126-127]{Schoen1989}.
\end{proof}

\section{Point singularity}\label{s-point}

In this section, we want to prove Theorem \ref{t-point-removable} and Proposition \ref{p-point}.\vskip.1cm

Let $M^3$ be a compact manifold without boundary so that the Yamabe invariant $\sigma=\sigma(M)$ of $M$ is negative. To be precise, let $\mathfrak{h}$ be a smooth background metric. Throughout this section, we  assume
  $g$ is a metric on $M$ so that:
 \begin{enumerate}
    \item [(i)] $g\in C^\infty(M\setminus A)\cap L^\infty(M)$ for some finite set $A$ so that
$$
\Lambda^{-1}\mathfrak{h}\le g\le \Lambda\mathfrak{ h}
$$
for some $\Lambda>0$;
\item[(ii)] $V_g(M)=1$.   
  \end{enumerate}
 \vskip .2cm
 For simplicity, we assume $A=\{p\}$, just one point.

 We follow the steps in \cite[\S6]{LiMantoulidis2019}. The first step is to 'blowing up' the singularity. By Lemma \ref{l-Green}, we have:
 \begin{lma}\label{l-green} For any $\eta>0$,
there is a positive smooth function $G$ on $M\setminus\{p\}$ such that
$$
(-8\Delta_g+\eta)G=0,
$$
and $G(x)\sim d_\mathfrak{h}(p,x)^{-1}$ near $p$.
\end{lma}

Assume $U$ is a local coordinate neighbourhood of $p$ given by $y\in B_1(0)$ so that $p$ corresponds to the origin, where $B_1(0)=\{y\in \R^3|\ |y|<1\}$.
Let $y=y(x)=x/|x|^2$ which maps $\R^3\setminus B_1(0) $ to $B_1(0)\setminus\{0\}$ so that $\R^3\setminus B_1(0)$ forms a coordinate  chart for $U\setminus\{p\}$ with coordinates $\{x^1,x^2,x^3\}$.

\begin{lma}\label{l-conformal} Let $\eta>0$ and $G$ be as in the previous lemma. For any  $\e>0$, let
\bee
g_\e =(1+\e G)^4g
\eee
which is a complete metric defined on $M\setminus \{p\}$.
\begin{enumerate}
  \item [(i)] The scalar curvature of $g_\e$ satisfies:
  $$
  \cR(g_\e)=(1+\e G)^{-5}\lf((\cR(g)-\eta)(1+\e G)+\eta\ri).
  $$
  \item[(ii)] In $U\setminus\{p\}$ with coordinates in $x\in \R^3\setminus B_1(0)$,
  $$
  C^{-1}|x|^{-4}\delta\le g\le C|x|^{-4}\delta
  $$
  for some positive constant $C$, where $\delta$ is the Euclidean metric.

  \item [(iii)] In $U\setminus\{p\}$ with coordinates in $x\in \R^3\setminus B_1(0)$,
  $$
  C_1\lf(|x|^{-1}+C_1\e\ri)^4\delta\le g_\e\le   C_2\lf(|x|^{-1}+C_2 \e\ri)^4\delta
  $$
  for some positive constants $C_1, C_2$ independent of $\e$. In particular, for $|x|\ge \e^{-1}$,
  $$
  C_3^{-1}\e^4\delta \le g_\e\le C_3\e ^4 \delta.
  $$
\end{enumerate}

\end{lma}
\begin{proof}
(i)
\bee
\begin{split}
\cR(g_\e)=&(1+\e G)^{-5}\lf[-8\Delta_g+\cR(g)\ri](1+\e G)\\
=&(1+\e G)^{-5}\lf(- \eta\e G+\cR(g)(1+\e G)\ri)\\
=&(1+\e G)^{-5}\lf( (\cR(g)-\eta)(1+\e G)+\eta\ri).
\end{split}
\eee

(ii) We may assume that  in $B_1(0)$
 $$
\mathfrak{ h}_{\a\b}\sim C\delta_{\a\b}
 $$
where $\mathfrak{h}_{\a\b}=\mathfrak{h}(\p_{y^\a},\p_{y^\b})$. Hence
\bee
\begin{split}
g_{ij}a^ia^j=&a^ia^j\frac{\p y^\a}{\p x^i}\frac{\p y^\b}{\p x^j}h_{\a\b}\\
=&h\left(a^i\frac{\p y^\a}{\p x^i}\p_{y^\a},a^j\frac{\p y^\b}{\p x^j}\p_{y^\b}\right)\\
\sim& a^ia^j\frac{\p y^\a}{\p x^i}\frac{\p y^\b}{\p x^j}\delta_{\a\b}\\
=&\sum_{i,j,\a}a^ia^j\frac{\p y^\a}{\p x^i}\frac{\p y^\a}{\p x^j}\\
=&\sum_{i,j,\a}a^ia^j\lf(\frac{\delta_{\a i}}{|x|^2}-2\frac{x^\a x^i}{|x|^4}\ri)\lf(\frac{\delta_{\a j}}{|x|^2}-2\frac{x^\a x^j}{|x|^4}\ri)\\
=&\frac{|a|^2}{|x|^4},
\end{split}
\eee
Hence (ii) is true.

(iii) For $y\in B_1(0)$, $G(y)\sim |y|^{-1}$. Hence in $x\in \R^3\setminus B_1(0)$,
$G(x)\sim |x|$. The results follow from (ii).

\end{proof}

The next step is to find a stable minimal sphere near infinity with respect to $g_\e$. In case the scalar curvature is positive, one can find a area minimizing set containing $B_1(0)$. Then any component is a sphere by the stability condition because the scalar curvature is positive. In our case, the scalar curvature may be negative, and we will use the result of Meeks-Yau \cite{MeeksYau1980} instead. Moreover, we need to control the volume inside the constructed minimal sphere.  We proceed as follows.\vskip .1cm

Let $0\le \phi\le 1$ be a non-increasing smooth function on $[0,\infty)$ so that $\phi=1$ on $[0,\frac32]$ and $\phi=0$ outside $[0,2]$.

\begin{lma}\label{l-convex} Let $\e>0$ and let $g_\e =(1+\e G)^4g$ be as in the previous lemma. Let  $R>\e^{-1}$ and
$$
\wh g_\e(x)=\phi(|x|/R)g_\e(x)+(1-\phi(|x|/R))s(x)
$$
for $x\in \R^3\setminus B_1(0)$ where $s(x)=(|x|^{-1}+\e)^4\delta$. Then $\wh g_\e$ can be extended to a smooth metric on $M\setminus\{p\}$ so that $\wh g_\e=g_\e$ outside $\R^3\setminus B_{R}(0).$ Moreover, $\p B_{2R}(0)$ is strictly convex with respect to $\wh g_\e$.
\end{lma}
\begin{proof} The first assertion follows from the fact that $\phi(|x|/R)=1$ for $|x|\le R$.
Near $|x|=2R$, $\wh g_\e= (|x|^{-1}+\e)^4\delta$. Hence
The mean curvature of $|x|=2R$ is given by
\bee
\begin{split}
H=&(|2R|^{-1}+\e  )^{-2}\left(\frac 1R+\frac4{(2R)^{-1}+\e}(-\frac1{4R^2})\ri)\\
=&(|2R|^{-1}+\e  )^{-2}R^{-1}\lf(1-\frac{1}{\frac12+\e R}\ri)\\
\ge&(|2R|^{-1}+\e  )^{-2}R^{-1}\lf(1-\frac{1}{\frac12+1}\ri)\\
>&0
\end{split}
\eee
  because $R\e>1$. Since   $|x|=2R$ is umbilical, hence $|x|=2R$ is strictly convex in $\wh g_\e$.
\end{proof}

\begin{lma}\label{l-msphere} For any $\e>0$, there is an embedded minimal sphere $S$ with respect to $g_\e$ such that the following are true.
\begin{enumerate}
  \item [(i)] There is $k>0$ independent of $\e$ such that  for any $r>1$, $S\subset  B_{k\e^{-1}}(0)\setminus B_r(0)$ provided $\e$ is small enough.
  \item[(ii)] $S$ is stable.
  \item[(iii)] $S$ bounds $B_1(0)$.
\item[(iv)]
$$
|V_{g_\e}(B_{k\e^{-1}}(0)\setminus B_1(0))-V_{g}(B_{k\e^{-1}}(0)\setminus B_1(0))|\le C\e|\log\e|
$$
for some $C>0$ independent of $\e$.
\end{enumerate}

\end{lma}
\begin{proof} For any $\e>0$, let $k\ge 1$ to be determined later, which will be independent $\e$. Let $R>k\e^{-1}$. Consider the metric $\wh g_\e$ as in the previous lemma. Since $\p B_{2R}(0)$ is strictly convex with respect to $\wh g_\e$. Denote
$$M_R=(M\setminus U) \cup (B_{2R}(0)\setminus B_1(0)).
 $$
Roughly speaking, $M_R$ is the part in $M$ which is `inside' $\p B_{2R}(0)$.
 Then $M_R$ is a compact manifold with boundary $\p B_{2R}(0)$.

For $r>1$, $\p B_r(0)$ will not be homotopically trivial in $M_R$   otherwise $M$ will be simply connected, see \cite[Proposition 3.10]{Hatcher}. By Hamilton's program and Perelman's
work on the resolution of the Poincar\'e conjecture,   $M$ is homeomorphic to $\mathbb{S}^3$.  This implies   $\sigma(M)>0$, contradicting the assumption that $\sigma(M)<0$. Since $\p B_{2R}(0)$ is strictly convex with respect to $\wh g_\e$, by  \cite[Theorem 7]{MeeksYau1980}, there is a homotopically non-trivial conformal map $f:\mS^2\to M_R$ such that, either $f$ is an embedding or $f$ is a two-to-one covering map whose image is an embedded $\mathbb{RP}^2$. Moreover, $f$ minimizes area among all homotopically nontrivial maps from $\mS^2$ to $M_R$.

We first claim for any $R>r>1$, the surface must be inside $B_{2R}(0)\setminus B_r(0)$ provided $\e>0$ is small enough. For fixed $r$, $\wh g_\e$ will converge uniformly in $C^\infty$ norm to $g$ as $\e\to0$ in $M_R\setminus (B_{2R}(0)\setminus B_r(0))$.  By \cite[Lemma 1]{MeeksYau1980}, we conclude that there is a constant $C>0$ which is independent of $\e$ such that if the surface mentioned above meets $M_R\setminus (B_{2R}(0)\setminus B_r(0))$, then its area is larger than $C$. On the other hand, $\p B_{\e^{-1}}(0)$ is homotopically nontrivial, with area with respect to $\wh g_\e$ being bounded above by $C\e^2$ by Lemma \ref{l-conformal} and the definition of $\wh g_\e$. Hence the surface must be inside $B_{2R}(0)\setminus B_r(0)$ provided $\e>0$ is small enough which may depend on $r$.

We can rule out that the image of $f$ is an embedded $\mathbb{RP}^2$ because the surface is now sitting inside $B_{2R}(0)\setminus B_1(0)$ which is simply connected.

To conclude, $f:\mS^2\to M_R\setminus B_1(0)$ is an embedded two-sphere. Denote the embedded surface by $S$. By the properties of $f$, we know that $S$ is stable. Moreover, $S$ bounds $B_1(0)$ considered as a surface in $\R^3$ because $S$ is homotopically nontrivial.

For the time being, $S$ is a stable minimal surface with respect to $\wh g_\e$. Since inside $  B_R(0)$,  $\wh g_\e=g_\e$, we want to show that $S\subset B_R(0)$, provided $\e$ is small enough, and for suitable $k\ge 1$ which is independent of $\e$.

 Let $1<\rho_0<2R$ be the smallest $\rho $ so $S\subset B_\rho(0)$. We want to estimate $\rho_0$. We proceed  as in \cite{LiMantoulidis2019}.    Suppose $\rho_0>\e^{-1}$. Let $\Omega=M\setminus V$ where $V$ is the open set
 in $\R^3$ which is exterior of $S$.
 Then for any $t>\e^{-1}$ so that $\p B_t(0)$ meets $S$ transversally,  $\Omega\setminus B_t(0)$ is an open set with boundary $\a_t\cup\b_t$ where
 $$
 \a_t= \p\Omega\setminus B_t(0)=S\setminus B_t(0),\ \ \b_t= \p B_t(0)\cap\Omega.
 $$
$\a_t$ and $\b_t$ are surfaces with boundary $\Sigma_t=\p\Omega\cap \p B_t(0)$.
We claim that one can deform $\a_t$ smoothly outside $B_t(0)$ to   surfaces with $\wh g_\e$-area arbitrary close to the $\wh g_\e$-area of $\b_t$ while keeping   $S\cap B_t(0)$ fixed.

 Let $r$ be the Euclidean distance function from the origin in $\R^3$. Then $h=r\circ f$ is a smooth function on $\mathbb{S}^2$.   Each component of the set $D=\{h>t\} $ is a planar domain bounded by simple closed curves in $\mathbb{S}^2$. Consider the open set in $\R^3$ bounded by $f(\mathbb{S}^2)$ and the sphere $|x|=t$. The intersection $D'$ of this open set  with $|x|=t$ consists of planar domains. We want to prove that $f(D)$ can be deformed in $|x|>t$ to surfaces  with area arbitrarily close to the area of $\Omega'$ fixing the common boundary of $D'$ and $f(D)$. By considering each component of $\Omega$ separately, we may assume that $D$ is connected and is a planar domain bounded by $C_0, C_1,\dots, C_k$. The boundary of $D'$ consists of $f(C_0), f(C_1),\dots, f(C_k)$. We may assume that $C_0$ contains all other $C_i$, and $f(C_0)$ contains all other $f(C_i)$. Note that $D$ is connected does not imply that $D'$ is connected.
Let us first consider the case that $D'$ is also connected. Using the Kelvin transform with origin at a point outside $D'$, we may transform $|x|=t$ to a plane and $|x|>t$ to the half space bounded by this plane. Hence we may assume that $f$ maps $\{h>t\}$ to an open set in $\R^3_+=\{x_3>0\}$ and maps its boundary to the plane $\{x_3=0\}.$
Now, we have a diffeomorphism $F$ from $D$ to $D'$  which maps    $C_i$ to $f(C_i)$ so that $F=f $ on $C_i$. Define $\Phi(s,p)=(1-s)f(p)+sF(p)$ (as position vectors in $\R^3$) for $p\in D$. This gives the deformation we want.

Suppose $D'$ has several components. Then each component corresponding to some subdomain in $D$ and the image of $D$ are images of these subdomains connected by   cylinders. Since each cylinder can be deformed a surface with  arbitrarily small area, we may use the above technique to get the result.

By the minimizing property of $S$, we have
\be\label{e-replace-1}
|\a_t|_{\wh g_\e}\le |\b_t|_{\wh g_\e}.
\ee
On the other hand, using the fact that $\p B_r(0)$ is convex with respect $\delta$, we conclude that
\be\label{e-replace-2}
|\b_t|_\delta\le |\a_t|_\delta\le C\e^{-2}
\ee
because   $|\a_t|_{\wh g_\e}\le |S|_{\wh g_\e}\le C\e^2$ and     $$C\e^4\delta\le \wh g_\e\le C^{-1}\e^4\delta$$ on $|x|>\e^{-1}$ for some $C>0$.
Here and below $C$ will denote a positive  constant  independent of $\e$, $k$ and $R$ and it may vary from line to line.

By the isoperimetric inequality on Euclidean spheres, see \cite{Osserman,Chavel}, we have
\bee
\begin{split}
|\Sigma_t|^2_{\wh g_\e}\ge &\e^4|\Sigma_t|^2_\delta\\
\ge&\e^4 \cdot  |\b_t|_\delta t^{-2}(4\pi t^2-|\b_t|_\delta)\\
\ge&  \e^4 |\b_t|_\delta t^{-2}(4\pi t^2-C\e^{-2})\\
\ge& C\e^4|\b_t|_\delta \\
\ge&C|\b_t|_{\wh g_\e}\\
\ge &C |\a_t|_{\wh g_\e}\\
=&C|\p\Omega\setminus B_t(0)|_{\wh g_\e}
\end{split}
\eee
provided $t\ge k_1\e^{-1}$, where $k_1$ is independent of $\e$ and $R$.

 Let us use the co-area formula. Let $r(x)=|x|$. Then
$$
|\nabla r|_{\wt g_\e}\le C\e^{-2}|\nabla_0 r|_\delta=C\e^{-2}.
$$
for some $C$ independent of $\e, R$. Here $\nabla_0$ is the Euclidean gradient.
Then for $t>k_1\e^{-1}$,
\be\label{e-rho-est}
\begin{split}
|\Sigma_t|_{\wh g_\e}\ge& C|\p\Omega\setminus B_t(0)|_{\wh g_\e}^\frac12\\
=&C\lf(\int_t^{\rho_0}\int_{\Sigma_s}|\nabla^Tr|^{-1} dH_{\wh g_\e}^1 ds\ri)^\frac12\\
\ge &C\e \lf(\int_t^{\rho_0} |\Sigma_s|_{\wh g_\e} ds\ri)^\frac12.
\end{split}
\ee
In particular,
$$
\lf(\int_{k_1\e^{-1}}^{\rho_0} |\Sigma_s|_{\wh g_\e} ds\ri)^\frac12\le C\e^{-1}|\p\Omega\setminus B_t(0)|_{\wh g_\e}^\frac12\le C'
$$
for some $C'$ independent of $\e$. By the differential inequality \eqref{e-rho-est}, integrating from $k_1\e^{-1}$ to $\rho_0$
Hence for $t\ge k_1\e^{-1}$,
\bee
\e (\rho_0-k_1\e^{-1})\le 2\lf(\int_{k_1\e^{-1}}^{\rho_0} |\Sigma_s|_{\wh g_\e} ds\ri)^\frac12\le C.
\eee
This implies $\rho_0\le   k_2\e^{-1}$, for some $k_2>k_1$ independent of $\e$ provided $R>k_2\e^{-1}$. Hence for this choice of $R$, the minimal surface $S$ will satisfy the conditions in the lemma because inside $B_R(0)$, $\wh g_\e=g_\e$.

To prove (iv),
\bee
\begin{split}
 V_{g_\e}(B_{k\e^{-1}(0)}\setminus B_1(0))-V_g(B_{k\e^{-1}}(0)\setminus B_1(0))=& \int_{B_{k\e^{-1}}(0)\setminus B_1(0)} (\sqrt{g_\e}-\sqrt g )dx\\
 =& \int_{B_{k\e^{-1}}(0)\setminus B_1(0)} ((1+\e G)^6-1)\sqrt g dx\\
 \le &C\int_1^{k\e^{-1}} \sum_{i=1}^6(\e r)^i r^{-2} dr\\
 \le &C\e|\log \e|.
\end{split}
\eee
for some constant $C$ independent of $\e$.

\end{proof}

Using the result in \S\ref{s-eigenvalue} below and the method in \cite{MantoulidisSchoen2015}, see also \cite{CCM2018}, one can construct a metric on the upper hemisphere with scalar curvature bounded below by $\sigma$ so that the boundary is isometric to $g_\e|_S$ in the previous lemma and is minimal. However, one cannot obtain a good estimate for the volume. Hence we can only obtain Theorem \ref{t-point-removable}, which we now prove.

\begin{proof}[Proof of Theorem \ref{t-point-removable}].  By the result on removable point singularities by Smith-Yang \cite{SmithYang1992}, it is sufficient to prove that $g$ is Einstein away from the point singularities.
 For simplicity, we prove the case that $A$ consists of only one point $p$. The general case can be proved similarly. If $g$ is not Einstein away from $p$, then either $\cR(g)>\sigma$ somewhere in $M\setminus\{p\}$  or $\cR(g)=\sigma$   and $\mathring{\Ric}\not\equiv0$ there, where $\mathring{\Ric}$ is the traceless part of $\Ric$.

Suppose there is $x_0\neq p$ with $\cR(g)(x_0)>\sigma$. Then there exists $r_0>0$ so that $p\notin B_{x_0}(2r_0)$ and
\be\label{e-case1-1}
\cR(g)\ge \sigma+2a
\ee
in $B_{x_0}(2r_0)$ for some $a>0$. Let $\eta>0$ be a constant and let $G$ be a positive function so that
$$
(-8\Delta_g +\eta)G=0
$$
in $M\setminus \{p\}$ so that $G\sim d(\cdot,p)$ near $p$. Let $g_\e=(1+\e G)^4g$. By Lemma \ref{l-msphere} for $\e>0$ small enough, there is a stable minimal sphere $S_\e$ containing the point $p$ so that $S_\e\to p$. Moreover, $M\setminus S_\e$ consists of two components. Let $M_\e$ be the component which does not contain $p$, then by Lemma  \ref{l-msphere},
\be\label{e-case1-2}
V_{g_\e}(M_\e)\to 1
\ee
as $\e\to 0$.
We glue two copies of $M_\e$ along $S_\e$.
This is homeomorphic to $M\# M$ and the smooth structure is unique. By the construction of Miao \cite[Proposition 3.1]{Miao2002}, for any $\delta>0$, we can find a smooth metric $g_{\e;\delta}$ so that $ g_{\e;\delta}=g_\e$ outside a neighbourhood $U_\delta$ of $S_\e$ in the connected sum, $\cR(g_{\e;\delta})$ are uniformly bounded below independent of $\delta$, $  g_{\e;\delta}$ are uniformly close to $g_\e$ as $\delta\to0$, and $V_{g_\e}(U_\delta)\to 0$ as $\delta\to0$.
By Lemma \ref{l-Lebrun},
\be\label{e-case1-3}
\begin{split}
2|\sigma|^\frac 32=&|\sigma(M\#M)|^\frac32\\
\le&\int_{M\#M} \cR(g_{\e;\delta})_-^\frac32 dV_{g_{\e;\delta}}\\
\rightarrow &2\int_{M_\e} \cR(g_{\e})_-^\frac32 dV_{g_\e}
\end{split}
\ee
as $\delta\to0$.
Recall that
 $$
  \cR(g_\e)=(1+\e G)^{-5}\lf((\cR(g)-\eta)(1+\e G)+\eta\ri).
  $$
Then outside $B_{x_0}(2r_0)$, we have
$$
\cR(g_\e)\ge \sigma-\eta
$$
because $\sigma<0$ and $1+\e G>1$. Inside $B_{x_0}(2r_0)$,
$$
\cR(g_\e)\ge (1+\e G)^{-5}\lf((\sigma+2a-\eta)(1+\e G)+\eta\ri)\ge \sigma+a
$$
provided $\eta<a$ and $\sigma+a<0$. We may assume $a$ is small enough and we can choose $\eta$ small enough so these are true. Hence by \eqref{e-case1-3}
\bee
\begin{split}
|\sigma|^\frac 32\le &\int_{M_\e\setminus B_{x_0}(2r_0)}|\sigma-\eta|^\frac32 dV_{g_\e}
+ \int_{B_{x_0}(2r_0)}|\sigma+a|^\frac32 dV_{g_\e}\\
=&|\sigma-\eta|^\frac32V_{g_\e}( M_\e\setminus B_{x_0}(2r_0))+|\sigma+a|^\frac32V_{g_\e}( M_\e\setminus B_{x_0}(2r_0))\\
=&|\sigma-\eta|^\frac32V_{g_\e}( M_\e)+(|\sigma+a|^\frac32-|\sigma-\eta|^\frac32) V_{g_\e}(B_{x_0}(2r_0))\\
\to& |\sigma-\eta|^\frac32+(|\sigma+a|^\frac32-|\sigma-\eta|^\frac32)V_{g}(B_{x_0}(2r_0)).
\end{split}
\eee
as $\e\to0$, for fixed $\eta$ by Lemma \ref{l-msphere}. Now let $\eta\to0$, we have
\bee
\begin{split}
|\sigma|^\frac 32\le & |\sigma|^\frac32+(|\sigma+a|^\frac32-|\sigma|^\frac32)V_{g}(B_{x_0}(2r_0))\\
<&|\sigma|^\frac32
\end{split}
\eee
because $\sigma<0$. This is a contradiction. Hence we must have $\cR(g)=\sigma$ on $M\setminus \{p\}$.\bigskip

Suppose $\cR(g)=\sigma$ in $M\setminus\{p\}$ and $\mathring{\Ric}\not\equiv0$ in $M\setminus\{p\}$,  then there is $x_1$,  $r_1>0$, $b>0$ so that $p\notin B_{x_1}(2r_1)$ and
\be\label{e-case2-1}
|\mathring{\Ric}|_g\ge b
\ee
in $B_{x_1}(2r_1)$. We may also assume that $r(x)=d_g(x_1,x)$ is smooth in $B_{x_1}(2r_1)$.
Let $\phi$   be a smooth  function on $[0,\infty)$ with $\phi\ge 0$, $\phi=1$ on $[0,1]$ and $\phi=0$ on $[2,\infty)$ and such that $|\phi'|\le C_1$, with $C_1$ being an absolute constant. Let $$h(x)=\phi\lf(\frac{r(x)} {r_1}\ri)\mathring{\Ric}(g)(x).$$

As in \cite{ShiTam2018}, for  $|\tau|>0$ sufficiently small, let $\wt g_{\tau}=g+\tau h$, which is smooth in $M\setminus \{p\}$ and is equal to $g$ outside $B_{x_1}(2r_1)$.  Since $\tr_g(h)=0$, we have
\be\label{e-tau}
\left\{
  \begin{array}{ll}
    dV_{\wt g_\tau}=dV_g+O(\tau^2)\\
V_{\wt g_\tau}(M)=V_g(M)+O(\tau^2)=1+O(\tau^2)\\
\Lambda^{-1}g\le \wt g_\tau \le \Lambda g
  \end{array}
\right.
\ee
for some $\Lambda>0$ independent of $\tau$.

Moreover,
\bee
\begin{split}
\cR(\wt g_\tau)=&\cR(g)+\tau\div_g(\div_gh)-\tau\la h,\Ric\ra_g+O(\tau^2)\\
=&\sigma+\tau\la \nabla^2\phi,\mathring{\Ric}\ra_g -\tau \phi |\mathring {\Ric}|^2_g+O(\tau^2)\\
\end{split}
\eee
Fix $\tau<0$, we want to  apply Lemma \ref{l-msphere} to the metric $\wh g_\tau=(V_{\wt g_\tau}(M))^{-\frac23} \wt g_\tau$  and use argument similar to the previous case.
Outside $B_{x_1}(2r_1)$,
\bee
\cR(\wh  g_\tau)\ge \sigma-C_2\tau^2
\eee
for some $C_2$ independent of $\tau$ because $V_{\wt g_\tau}(M)=1+O(\tau^2)$. Inside $B_{x_1}(2r_1)$
\bee
\cR(\wh g_\tau)=\sigma+\tau\la \nabla^2\phi,\mathring{\Ric}\ra_g -\tau \phi |\mathring {\Ric}|^2_g+O(\tau^2)
\eee

Apply Lemma \ref{l-msphere} and argue as before, we have
\be\label{e-case2-2}
\begin{split}
|\sigma|^\frac32&\le |\sigma|^{\frac32}V_g( M\setminus B_{x_1}(2r_1))\\
&\quad +\int_{B_{x_1}(2r_1)}\lf((\sigma+\tau\la \nabla^2\phi,\mathring{\Ric}\ra_g -\tau \phi |\mathring {\Ric}|^2_g)_-\ri)^\frac 32 dV_g+O(\tau^2)
\end{split}
\ee

If $|\tau|$ is small enough, then
\bee
\begin{split}
&\quad \lf((\sigma+\tau\la \nabla^2\phi,\mathring{\Ric}\ra_g -\tau \phi |\mathring {\Ric}|^2_g)_-\ri)^\frac 32\\
&= (-\sigma-\tau\la \nabla^2\phi,\mathring{\Ric}\ra_g +\tau \phi |\mathring {\Ric}|^2_g )^\frac 32\\
&=|\sigma|^\frac32\left[1+\frac32\cdot\frac\tau\sigma \tau\la \nabla^2\phi,\mathring{\Ric}\ra_g -\frac32\frac\tau \sigma \phi |\mathring {\Ric}|^2_g\right]+O(\tau^2).
\end{split}
\eee
Hence
\bee
\begin{split}
&\int_{B_{x_1}(2r_1)}\lf((\sigma+\tau\la \nabla^2\phi,\mathring{\Ric}\ra_g -\tau \phi |\mathring {\Ric}|^2_g)_-\ri)^\frac 32 dV_g\\
=&|\sigma|^\frac32\int_{B_{x_1}(2r_1)}\left(1+\frac32\cdot\frac\tau\sigma  \la \nabla^2\phi,\mathring{\Ric}\ra_g -\frac32\frac \tau\sigma\phi   |\mathring {\Ric}|^2_g \right)dV_g+O(\tau^2).
\end{split}
\eee
Combine these with \eqref{e-case2-2}, we have
\bee
\begin{split}
|\sigma|^\frac32\le &|\sigma|^{\frac32}
+|\sigma|^\frac32\int_{B_{x_1}(2r_1)}\left( \frac32\cdot\frac\tau\sigma  \la \nabla^2\phi,\mathring{\Ric}\ra_g -\frac32\frac \tau\sigma \phi  |\mathring {\Ric}|^2_g\right) dV_g+O(\tau^2)\\
=&|\sigma|^{\frac32}-\frac32\frac{|\tau|}{|\sigma|}  \int_{B_{x_1}(2r_1)}(   \phi|\mathring {\Ric}|^2_g) dV_g+O(\tau^2)\\
\le &|\sigma|^{\frac32}-\frac32\frac{|\tau|}{|\sigma|}\cdot b V_g( B_{x_1}(r_1))\\
<&|\sigma|^\frac32.
\end{split}
\eee
if we  choose $|\tau|$ small enough. Hence we have a contradiction.
\end{proof}

Next we will prove Proposition \ref{p-point}.

\begin{proof}[Proof of Proposition \ref{p-point}] We may assume that the $\sigma(M)$   is negative, otherwise  the result is obvious. For simplicity, we assume $A$ consists of only one point $p$. Let $U$ be an open set in $M$ containing $p$ so that $\cR(g)>0$ in $U\setminus\{p\}$.  Let $0\le \phi\le 1$ be a smooth cut-off function near $p$ so that $\phi=0$ outside $U$ and $\phi=1$ on an open set $  V\Subset U$ with $p\in V$. Let $\kappa(x)=\rho(\frac14\cR(g)(x)\cdot \phi(x), 1)$ where $\rho$ is the function defined in \eqref{e-rho-1}. Then $\kappa$ is  a smooth function in $M\setminus\{p\}$  so that $\kappa=0$ outside $U$, $\kappa$ is bounded, and $\frac12\cR(g)>\kappa>0$ in $  V\setminus\{p\}$. In particular, there is $x_0$ and $r_0>0$ so that $B_{x_0}(r_0)\Subset V\setminus\{p\}$ and $\cR(g)\ge a$ for some $a>0$ in $B_{x_0}(r_0)$. By Lemma \ref{l-Green}, we can find a positive smooth function $G$ so that
$$
-8\Delta_gG+\kappa G=0
$$
in $M\setminus\{p\}$, $G$ is bounded below by a positive constant and $G(x)\sim \left( r(p,x)\right)^{-1}$ near $p$. As in the proof of previous theorem, consider the metric
$$
g_\e=(1+\e G)^4g
$$
Then the scalar curvature of $g_\e$ is
$$
\cR(g_\e)=(1+\e G)^4\lf((\cR(g)-\kappa)(1+\e G)+\kappa\ri).
$$
In particular, the $\cR(g_\e)>0$ in $V\setminus\{p\}$. Then as in \cite{LiMantoulidis2019} or using Lemma \ref{l-msphere}, we can find a minimal surface $S_\e$ which is a sphere and bounds a  three ball $W$ containing the point $p$. Moreover, the first eigenvalue of $-\Delta_S +K$ is positive because for $\e$ small enough $S_\e$ will be inside  $V\setminus\{p\}$ and $\cR(g_\e)>0$ there. As in \cite{LiMantoulidis2019}, using the method in \cite{MantoulidisSchoen2015}, we can glue the part $M\setminus W$ along $S_\e$ with a hemisphere so that the scalar curvature at in the hemisphere is nonnegative. Using the method \cite{Miao2002} as in the proof of Theorem \ref{t-point-removable} and let $\e\to \infty$, we have
\bee
|\sigma(M)|^\frac23\le |\sigma|^\frac32V_g(M\setminus B_{x_0}(r_0))<|\sigma|^\frac23
\eee
by Lemma \ref{l-Lebrun}. Hence $0>\sigma(M)>\sigma$. From this the result follows.
\end{proof}

\section{On the eigenvalues of $-\Delta+K$}\label{s-eigenvalue}

In \cite{LiMantoulidis2019} and the proof in Proposition \ref{p-point} the following fact in \cite{MantoulidisSchoen2015} is used: Suppose $g$ is a metric on $\mathbb{S}^2$ so that the first eigenvalue of the operator  $\lambda_1(-\Delta_g+K(g))$ is positive, then $g$ can be connected smoothly to a the standard metric on $\mathbb{S}^2$ by a family of metrics $g(s)$ so that $\lambda_1(-\Delta_{g(s)}+K(g(s)))>0$ for $s>0$. In \cite{ChauMartens}, Chau-Martens prove that this is still true for metric $g$ with $\lambda_1(-\Delta_g+K)=0$. From this one can construct a smooth manifold with boundary with positive scalar curvature so that the boundary is $(\mathbb{S}^2,g)$ and is minimal.

One may ask what one can say about the stable minimal surface constructed in the proof of Theorem \ref{t-point-removable}. In this case,  One want to consider the situation so that $\lambda_1(-\Delta_g+K)\ge \frac\sigma 2$ which is negative. We want to prove that a similar result is true in this case. To be precise,  we want to prove the following (see also  \cite{LM-2021}):

\begin{prop}\label{p-eigenvalue}
Let $g_0$ be a metric on $\mathbb{S}^2$ so that the first eigenvalue of $-\Delta_{g_0}+K$ is bounded below by $\a$ with $\a\le 0$ where $K$ is the Gaussian curvature of $g_0$. Then there is a smooth family of metrics $ g(s)$, $s\in [0,1]$ so that {$g(0)=g_0$, $g(1)$ is the spherical metric with area $A_0$ and} the first eigenvalue $\lambda_1(s)$ of
$-\Delta_{g(s)}+K(g(s))$ satisfies $\lambda_1(s)>\a$ for $s>0$. Moreover, the area $A(s)$ with respect to $g(s)$ are constant in $s$, namely $A(s)=A_0$ which is the area with respect to $g_0$.
\end{prop}

{ We will construct $g(s)$ by using the normalized Ricci flow.} The proposition follows from the monotonicity result of J.-F. Li \cite{Li2007} and the well-known results on Ricci flows on surfaces, see \cite{Hamilton1988,Chow1991}. Using the computations in \cite{Li2007}, we have:

\begin{thm}\label{t-eigenvalue} Let $g(t)$ be a solution to the normalized Ricci flow on a compact manifold $M^n$ with dimension $n$:
\bee
\frac{\p}{\p t}g=-2\Ric+\frac2n \mathfrak{s} g
\eee
where $\mathfrak{s}$ is the average of the  scalar curvature $\mathcal{S}(g(t))$. Let $\lambda(t)$ be the first eigenvalue of the operator $L(t)=-\Delta_{g(t)}+\frac 12 \mathcal{S}$. Then
$$
\frac{d\lambda}{dt}+ \frac{2\mathfrak{s}}n\lambda=\int_M(\frac12|\Ric+\nabla^2 u|^2f^2+ \frac12|\Ric|^2 f^2) dV_{g(t)}
$$
where $f>0$ is the first eigenfunction of $L(t)$ normalized by $\int_M f^2 dV_{g(t)}=1$ and $u$ is such that $f^2=e^{-u}$.
\end{thm}

\begin{proof}[Proof of Proposition \ref{p-eigenvalue}] In case of surface, $\mathcal{S}=2K$ where $K$ is the Gaussian curvature, and $\Ric=Kg$. Moreover, by   Gauss-Bonnet theorem
$$
\mathfrak{s}=\frac{4\pi(1-\mathfrak{g})}{A_0}
$$
is a constant where $\mathfrak{g}$ is the genus of the surface and $A_0$ is the area of $g_0$ because the area will be constant along the normalized Ricci flow. Hence using the notation in Theorem \ref{t-eigenvalue},
we have for $M=\mathbb{S}^2$
$$
\frac{d}{dt}(e^{4\pi A_0^{-1} t}\lambda)=e^{4\pi A_0^{-1}t}\int_M(\frac12|\Ric+\nabla^2 u|^2f^2+ \frac12|\Ric|^2 f^2)dV_{g(t)}.
$$
Since $\int_{M}K^2f^2 dV_{g_(t)}>0$ for all $t$, and
since the normalized Ricci flow has long time solution and will converge to a metric with constant curvature by  \cite{Hamilton1988,Chow1991}.
\bee
\lambda(t)> e^{-4\pi A_0^{-1}t}\lambda(0)\ge \a
\eee
for all $t>0$ because $\a\le0$ and $0<e^{-4\pi A_0^{-1}t}<1$. From these, the result follows { by re-parametrizing the time.}
\end{proof}

By  \cite{CCM2018},  using this one can show that  if $g$ is a metric on $\mathbb{S}^2$ so that $\lambda_1(-\Delta_g+K(g))\ge \frac\sigma2$, with $\sigma<0$, then one can find a metric on the upper hemisphere of $\mathbb{S}^3$ with metric $h$ so that the scalar curvature of $h$ is at least $\sigma$ and so that $h=g$ when restricted to the equator which is minimal. Hence we have the following   result    of Cabrera Pacheco,  Cederbaum, and McCormick  \cite[Theorem 4.1]{CCM2018}:

\begin{cor}\label{c-eigenvalue} Let $\Sigma$ be a two sphere with metric $g$ with $|\Sigma|_g=A$ so that $\lambda_1(-\Delta_g+K)\ge-3 $, then for any $m>\frac12\lf((\frac{A}{4\pi})^\frac12+(\frac{A}{4\pi})^\frac32\ri)$, $(\Sigma,g)$ can be embedded as a minimal surface in an asymptotically hyperbolic manifold with scalar curvature bounded from below by $-6$ so that the manifold is the ADS-Schwarzschild manifold with mass $m$ and is foliated by so that mean convex spheres.
\end{cor}
In the original version, it was assumed that $\lambda_1(-\Delta_g+K(g))>0 $ or $K(g)>-3$. Using Proposition \ref{p-eigenvalue}, one can see that the corollary is true under the assumption  $\lambda_1(-\Delta_g+K(g))\ge-3 $. See also the remark on \cite[p.347]{CCM2018}.

\appendix

\section{Existence of Green's function}
Let $(M^3,g)$ be a compact manifold so that $g$ is a Riemannian metric which is smooth away from some finite set $A$  of points. Moreover $g$ is uniformly equivalent to some  smooth background metric. We want to sketch a proof of the following:
\begin{lma}\label{l-Green}
Let $\kappa\ge0$ be a bounded function which is smooth away from $A$ and is positive on some open set containing $A$. Then there exists $u$ so that
$$
\Delta_g u-\kappa u=0
$$
in $M\setminus A$, $u\in C^\infty(M\setminus A)$, $u\ge C$ for some $C>0$, and $C^{-1}(r(x,p))^{-1}\le u(x)\le C(r(x,p))^{-1}$ near any point $p\in A$ for some $C>0$, where $r(x,p)$ is the distance function from $p$ with respect to the background metric.
\end{lma}
\begin{proof} We assume that $A$ consists of only one point $p$. The general case can be proved similarly. Let $U$ be an precompact connected open set of $p$ with smooth boundary. We may assume that $\kappa>0$ inside $U$. By \cite{LittmanStampacchiaWeinberger}, there is $v$ satisfying
$$
\Delta_gv=0
$$
in $U\setminus\{p\}$ so that $v\in C^\infty(U\setminus\{p\})$ and $C^{-1}(r(x,p))^{-1}\le v(x)\le C(r(x,p))^{-1}$ and with zero boundary data on $\p U$. By \cite[Theorem 8.3]{GilbargTrudinger}, we can find a $W^{1,2}(U)$ solution with zero boundary data of
$$
\Delta_gw-\kappa w=-\kappa v.
$$
By  \cite[ Theorem 8.16]{GilbargTrudinger}, $w$ is bounded. Moreover $w$ is smooth away from $p$. Extend $v-w$ to be a smooth function on $M\setminus \{p\}$ so that it is equal to $v-w$ near $p$. Denote this function by $\phi$ and let $\psi=\Delta_g \phi-\kappa \phi$. Then $\psi=0$ near $p$ and it is smooth. We want to find $f$ so that $f$ is smooth in $M\setminus \{p\}$, $f$ is bounded and
$$
\Delta_g f-\kappa f=\psi
$$
where $\psi:=\Delta_g\phi-\kappa\phi.$
Suppose $f$ is found. Let $u=\phi-f$. Then $u$ is smooth away from $p$, $\Delta_g u-\kappa u=0$ and $C^{-1}(r(x,p))^{-1}\le u(x)\le C(r(x,p))^{-1}$ near $p$. By { adding a large constant,} we can have $u\ge C>0$ for some constant $C$.

It remains to find $f$. Let $\{\Omega_k\}$ be an increasing sequence of   domains  with smooth boundary which exhaust  $M\setminus\{p\}$. Let $f_k$ be solution of
\bee
 \left\{
   \begin{array}{ll}
     \Delta_g f_k-\kappa f_k=\psi & \hbox{in $\Omega_k$;} \\
     f_k=0, & \hbox{on $\p\Omega_k$.}
   \end{array}
 \right.
\eee
We may assume that $\psi=0$ outside $\Omega_1$. We may also assume $\kappa>0$ outside $\Omega_1$.  Let $ 0\le  \theta\le 1$ be a cutoff function so that $\theta=1$ in $\Omega_1$ and is $0$ outside $\Omega_2$.  Then we have
\bee
\begin{split}
\int_{\Omega_k}(|\nabla f_k|^2+\kappa f_k^2) dV_g=&-\int_{\Omega_1}f_k\psi dV_g\\
=&-\int_{\Omega_2}\theta f_k\psi dV_g \\
=&-\int_{\Omega_2}  \theta f_k (\Delta_g\phi-\kappa\phi) dV_g\\
\le &C_1\lf[\int_{\Omega_2\setminus \Omega_1}(|\nabla f_k|+\kappa|f_k|)dV_g+\int_{\Omega_2}\kappa |f_k| dV_g\ri]\\
\le &C_2\lf(\e\int_{\Omega_2}(|\nabla f_k|^2+\kappa f_k^2) dV_g+ \e^{-1}C_3\ri)
\end{split}
\eee
for $k$ large enough, because $\psi=0$ outside $\Omega_1$ and $\theta=1$ in $\Omega_1$. Moreover, we have used the fact that $0<\kappa\le C$ for some positive constant in $\Omega_2\setminus\Omega_1$.
Here and below $C_i$ denote positive constants which are independent of $k$. Hence we conclude that
$$
\int_{\Omega_k}(|\nabla f_k|^2+\kappa f_k^2) dV_g\le C_3.
$$
for all $k$. In particular, the $W^{1,2}(\Omega_2\setminus\Omega_1)$ of $f_k$ is uniformly bounded because $\kappa>C'$ for some $C'>0$ in $(\Omega_2\setminus\Omega_1)$. By  \cite[Theorem 8.15]{GilbargTrudinger}, we conclude that $f_k$ are locally uniformly bounded in $\Omega_2\setminus\Omega_1$. Since $f_k=0$ on $\p\Omega_k$ and $\psi=0$ outside $\Omega_1$,  by the maximum principle we conclude that $f_k$ are uniformly bounded by a constant independent of $k$. From this, it is easy to see that $f$ can be found.
\end{proof}


\end{document}